\documentclass{proc-l}

\newtheorem{theorem}{Theorem}[section]
\newtheorem{lemma}[theorem]{Lemma}

\newtheorem{proposition}[theorem]{Proposition}

\theoremstyle{definition}
\newtheorem{definition}[theorem]{Definition}
\newtheorem{example}[theorem]{Example}

\theoremstyle{remark}
\newtheorem{remark}[theorem]{Remark}

\numberwithin{equation}{section}

\begin{document}

\title{Generalized discrete operators}

\author{Rui A. C. Ferreira}
\address{Grupo F\'isica-Matem\'atica, Faculdade de Ci\^encias, Universidade de Lisboa,
Av. Prof. Gama Pinto, 2, 1649-003, Lisboa, Portugal.}
%\curraddr{Department of Mathematics and %Statistics,
%Case Western Reserve University, Cleveland, %Ohio 43403}
\email{raferreira@fc.ul.pt}
\thanks{Rui A. C. Ferreira was supported by the ``Funda\c{c}\~{a}o para a Ci\^encia e a Tecnologia (FCT)" through the program ``Stimulus of Scientific Employment, Individual Support-2017 Call" with reference CEECIND/00640/2017.}

%    General info
\subjclass[2000]{Primary 39A12 , 26A33}
\date{}

\keywords{Discrete calculus, fractional operators.}

\begin{abstract}
We define a class of discrete operators that, in particular, include the delta and nabla fractional operators.  Moreover, we prove the fundamental theorem  of  calculus for these operators.
\end{abstract}

\maketitle

\section{Preamble}\label{Preamble}

The theory of discrete fractional calculus is currently  an area of mathematics of intensive research, having appeared in the literature many articles on the subject in the past decade (see \cite{Abdeljawad2,Atici1,Cermak,Ferreira1,Ferreira3,Ferreira2,Gray,MillerRoss} and the references therein). Two parallel concepts were introduced, namely, the delta (or forward) operators and the nabla (or backward) operators (see \cite[Sections 2 and 3]{GoodrichBook}).

Consider the falling function defined, for $x,y\in A\subset\mathbb{R}$, by
$$x^{\underline{y}}=\left\{\begin{array}{llll}x(x-1)\ldots(x-y+1)\ \mbox{for}\ y\in\mathbb{N}_1, \\
1\ \mbox{for } y=0,\\
\frac{\Gamma(x+1)}{\Gamma(x+1-y)}\ \mbox{for }x,x-y\notin\mathbb{N}^{-1},\\
0\ \mbox{for }x\notin\mathbb{N}^{-1}\mbox{ and }x-y\in\mathbb{N}^{-1},
    \end{array}\right.
$$
and the rising function defined by 
\begin{equation}\label{rfunction}
x^{\overline{y}}=(x+y-1)^{\underline{y}}.
\end{equation}
Then, letting $\mathbb{N}_a=\{a,a+1,\ldots\}$ with $a\in\mathbb{R}$ and $f:\mathbb{N}_a\to\mathbb{R}$ being a function, the delta Riemann--Liouville fractional sum of $f$ of order $\nu>0$ is defined by 
$$\Delta_{a}^{-\nu} f(t)=\sum_{s=a}^{t-\nu}\frac{(t-(s+1))^{\underline{\nu-1}}}{\Gamma(\nu)}f(s)\mbox{ for } t\in\mathbb{N}_{a+\nu-1},$$
while the nabla Riemann--Liouville fractional sum of $f$ of order $\nu>0$ is defined by 
$$\nabla_{a}^{-\nu} f(t)=\sum_{s=a+1}^{t}\frac{(t-(s-1))^{\overline{\nu-1}}}{\Gamma(\nu)}f(s)\mbox{ for } t\in\mathbb{N}_{a}.$$
One can observe that the above sums are of the type 
$$\sum k(t-s\pm 1)f(s),$$
for a certain kernel function $k$. Now, if we consider the delta difference operator $\Delta f(t)=f(t+1)-f(t)$ and the nabla difference operator $\nabla f(t)=f(t)-f(t-1)$, then the delta and nabla Riemann--Liouville fractional differences of $f$ of order $0<\alpha\leq 1$ are defined by, 

$$\Delta_{a}^{\alpha} f(t)=\Delta [\Delta_{a}^{-(1-\alpha)}f](t),\ t\in\mathbb{N}_{a+1-\alpha},$$ 
and $$\nabla_{a}^{\alpha} f(t)=\nabla [\nabla_{a}^{-(1-\alpha)}f](t),\ t\in\mathbb{N}_{a+1},$$
respectively.

In this work we aim to construct a summation and a difference operator generalizing the above ones and satisfying the fundamental theorem of calculus (we are particularly inspired by the work of Kochubei \cite{Kochubei} in which such kind of operators were defined for (continuous) integrals and derivatives). Hopefully, these very general operators will be useful for researchers acting within the discrete calculus theory.

\section{Main results}

Let us start by recalling the discrete convolution of two functions $f,g:\mathbb{N}_a\to\mathbb{R}$, with $a\in\mathbb{R}$: it is denoted by $(f*g)_a$ and defined by
$$(f*g)_a(t)=\sum_{\tau=a}^{t-1}f(t-\tau-1+a)g(\tau),\quad t\in\mathbb{N}_a.$$
Here and throughout this text we assume that empty sums are equal to zero. Therefore, $(f*g)_a(a)=0$ for all functions $f$ and $g$. It is known that the convolution is commutative and associative (cf. \cite[Theorem 5.4]{Bohner}).

Let us introduce the following set of pair-of-functions: For $a\in\mathbb{R}$, put $$\mathcal{C}_a=\{p,q:\mathbb{N}_a\to\mathbb{R}:(p*q)_a(t)=1\ \mbox{for all }t\in\mathbb{N}_{a+1}\}.$$

Before we proceed, we state here the fractional power rule, whose proof may be found in 
\cite{Ferreira2}.

\begin{lemma}\label{lem1}
Let $a\in\mathbb{R}$. Assume $\mu\notin\mathbb{N}^{-1}$ and $\nu>0$. If $\mu+\nu\notin\mathbb{N}^{-1}$, then
\begin{equation}\label{FPS11}
    \Delta_{a+\mu}^{-\nu}[(s-a)^{\underline{\mu}}](t)=\frac{\Gamma(\mu+1)}{\Gamma(\mu+\nu+1)}(t-a)^{\underline{\mu+\nu}},\mbox{ for } t\in\mathbb{N}_{a+\mu+\nu}.
\end{equation}
\end{lemma}

\begin{example}
Let $a\in\mathbb{R}$ and $\alpha\in\mathbb{R}^+\backslash\mathbb{N}_1$. Define the functions
$$p(t)=\frac{(t-a)^{\underline{\alpha-1}}}{\Gamma(\alpha)}\mbox{ and } q(t)=\frac{(t-a+1-2\alpha)^{\underline{-\alpha
}}}{\Gamma(1-\alpha)},\quad t\in\mathbb{N}_{a+\alpha-1}.$$
Then, for all $t\in\mathbb{N}_{a+\alpha}$, we have
\begin{align*}
    (p*q)_{a+\alpha-1}(t)&=\sum_{\tau=a+\alpha-1}^{t-1}\frac{(t-\tau-1+\alpha-1)^{\underline{\alpha-1}}}{\Gamma(\alpha)}\frac{(\tau-a+1-2\alpha)^{\underline{-\alpha}}}{\Gamma(1-\alpha)}\\
    &=\sum_{\tau=a}^{t-\alpha}\frac{(t-(\tau+1))^{\underline{\alpha-1}}}{\Gamma(\alpha)}\frac{(\tau-(a+\alpha))^{\underline{-\alpha}}}{\Gamma(1-\alpha)}\\
    &=\frac{ \Delta_{a}^{-\alpha}[(\tau-(a+\alpha))^{\underline{-\alpha}}](t)}{\Gamma(1-\alpha)}\\
    &=1,
\end{align*}
where we used \eqref{FPS11} to obtain the last equality. Hence, the pair $(p,q)\in\mathcal{C}_{a+\alpha-1}$.

Analogously, if we define the functions 
$$\hat{p}(t)=\frac{(t-a+1)^{\overline{\alpha-1}}}{\Gamma(\alpha)}\mbox{ and } \hat{q}(t)=\frac{(t-a+1)^{\overline{-\alpha}}}{\Gamma(1-\alpha)},$$
then we may show, upon using \eqref{rfunction} and Lemma \ref{lem1}, that $(\hat{p},\hat{q})\in\mathcal{C}_a$. Indeed,
\begin{align*}
    (\hat{p}*\hat{q})_{a}(t)&=\sum_{\tau=a}^{t-1}\frac{(t-\tau)^{\overline{\alpha-1}}}{\Gamma(\alpha)}\frac{(\tau-a+1)^{\overline{-\alpha}}}{\Gamma(1-\alpha)}\\
    &=\sum_{\tau=a}^{t-1}\frac{(t+\alpha-1-(\tau+1))^{\underline{\alpha-1}}}{\Gamma(\alpha)}\frac{(\tau-(a+\alpha))^{\underline{-\alpha}}}{\Gamma(1-\alpha)}\\
    &=\frac{ \Delta_{a}^{-\alpha}[(\tau-(a+\alpha))^{\underline{-\alpha}}](t+\alpha-1)}{\Gamma(1-\alpha)}\\
    &=1,
\end{align*}

\end{example}

\begin{remark}
A somewhat more elaborate example of a pair of functions belonging to $\mathcal{C}_0$ may be found in \cite[Example 9]{Goodrich}.
\end{remark}

\begin{definition}
Let $p,q:D\subset\mathbb{R}\to\mathbb{R}$ be two functions. We define the generalized fractional sum (GFS) of $f:\mathbb{N}_a\to\mathbb{R}$ by
$$\mathcal{S}_{\{(p);a\}}f(t)=\sum_{\tau=a}^{t-1}p(t-\tau-1+a)f(\tau),\quad t\in\mathbb{N}_a.$$
Moreover, the generalized fractional difference of Riemann--Liouville type is defined by
$$^{RL}\mathcal{D}_{\{(q);a\}}f(t)=\Delta\sum_{\tau=a}^{t-1}q(t-\tau-1+a)f(\tau),\quad t\in\mathbb{N}_{a+1},$$
while the generalized fractional difference of Caputo type is defined by
$$^{C}\mathcal{D}_{\{(q);a\}}f(t)=\sum_{\tau=a}^{t-1}q(t-\tau-1+a)\Delta f(\tau),\quad t\in\mathbb{N}_{a+1}.$$
\end{definition}

The previous definition includes the delta and nabla operators mentioned in Section \ref{Preamble}. Indeed, first consider a function $f:\mathbb{N}_a\to\mathbb{R}$. Then,

\begin{itemize}
    \item Consider $p(t)=\frac{(t-a)^{\underline{\alpha-1}}}{\Gamma(\alpha)}$, for $t\in\mathbb{N}_{a+\alpha-1}$. Then, for $t\in\mathbb{N}_{a+\alpha-1}$, we have
    \begin{align*}
        \mathcal{S}_{\{(p);a+\alpha-1\}}[f(\tau+1-\alpha)](t)&=\sum_{\tau=a+\alpha-1}^{t-1}\frac{(t-(\tau+1)+\alpha-1)^{\underline{\alpha-1}}}{\Gamma(\alpha)}f(\tau+1-\alpha)\\
        &=\sum_{\tau=a}^{t-\alpha}\frac{(t-(\tau+1))^{\underline{\alpha-1}}}{\Gamma(\alpha)}f(\tau)\\
        &=\Delta_a^{-\alpha}f(t).
    \end{align*}
    
    \item Consider $q(t)=\frac{(t-a)^{\underline{-\alpha}}}{\Gamma(1-\alpha)}$, for $t\in\mathbb{N}_{a-\alpha}$. Then, for $t\in\mathbb{N}_{a+1-\alpha}$, we have
    \begin{align*}
        ^{RL}\mathcal{D}_{\{(q);a+\alpha-1\}}[f(\tau+\alpha)](t)&=\Delta\sum_{\tau=a-\alpha}^{t-1}\frac{(t-(\tau+1)-\alpha)^{\underline{-\alpha}}}{\Gamma(1-\alpha)}f(\tau+\alpha)\\
        &=\Delta\sum_{\tau=a}^{t+\alpha-1}\frac{(t-(\tau+1))^{\underline{-\alpha}}}{\Gamma(1-\alpha)}f(\tau)\\
        &=\Delta_a^{\alpha}f(t).
    \end{align*}
\end{itemize}

Now, consider a function $f:\mathbb{N}_{a+1}\to\mathbb{R}$. Then,
\begin{itemize}
    \item Consider $\hat{p}(t)=\frac{(t-a+1)^{\overline{\alpha-1}}}{\Gamma(\alpha)}$, for $t\in\mathbb{N}_{a}$. Then, for $t\in\mathbb{N}_{a}$, we have
    \begin{align*}
        \mathcal{S}_{\{(\hat{p});a\}}[f(\tau+1)](t)&=\sum_{\tau=a}^{t-1}\frac{(t-\tau)^{\overline{\alpha-1}}}{\Gamma(\alpha)}f(\tau+1)\\
        &=\sum_{\tau=a+1}^{t}\frac{(t-(\tau-1))^{\overline{\alpha-1}}}{\Gamma(\alpha)}f(\tau)\\
        &=\nabla_a^{-\alpha}f(t).
    \end{align*}
    
    \item Consider $\hat{q}(t)=\frac{(t-a+1)^{\overline{-\alpha}}}{\Gamma(1-\alpha)}$, for $t\in\mathbb{N}_{a}$. Then, for $t\in\mathbb{N}_{a+1}$, we have
    \begin{align*}
        ^{RL}\mathcal{D}_{\{(\hat{q});a+1\}}[f](t)&=\Delta\sum_{\tau=a+1}^{t-1}\frac{(t-\tau)^{\overline{-\alpha}}}{\Gamma(1-\alpha)}f(\tau)\\
        &=\sum_{\tau=a+1}^{t}\frac{(t+1-\tau)^{\overline{-\alpha}}}{\Gamma(1-\alpha)}f(\tau)-\sum_{\tau=a+1}^{t-1}\frac{(t-\tau)^{\overline{-\alpha}}}{\Gamma(1-\alpha)}f(\tau)\\
        &=\nabla_a^{\alpha}f(t).
    \end{align*}
\end{itemize}

To prove our main result we need the following (cf. \cite[Theorem 1.67]{GoodrichBook}):

\begin{proposition}[Leibniz rule]\label{leibnizrule}
Assume $f:\mathbb{N}_{a+1}\times\mathbb{N}_a\to\mathbb{R}$. Then,
$$\Delta_t\sum_{s=a}^{t-1}f(t,s)=\sum_{s=a}^{t-1}\Delta_t f(t,s)+f(t+1,t),\quad t\in\mathbb{N}_a.$$
\end{proposition}

It follows the fundamental theorem of calculus for these generalized discrete operators.

\begin{theorem}
Let $a\in\mathbb{R}$ and suppose that $(p,q)\in\mathcal{C}_a$. Then,
\begin{equation}\label{proof1}
    ^{RL}\mathcal{D}_{\{(q);a\}}\mathcal{S}_{\{(p);a\}}f(t)={^{C}\mathcal{D}}_{\{(q);a\}}\mathcal{S}_{\{(p);a\}}f(t)=f(t),\quad t\in\mathbb{N}_{a+1}.
\end{equation}
Moreover,
\begin{equation}\label{d1}
    \mathcal{S}_{\{(p);a\}}{^{RL}\mathcal{D}}_{\{(q);a\}}f(t)=f(t),\quad t\in\mathbb{N}_{a}
\end{equation}
and 
\begin{equation}\label{d2}
    \mathcal{S}_{\{(p);a\}}{^{C}\mathcal{D}}_{\{(q);a\}}f(t)=f(t)-f(a),\quad t\in\mathbb{N}_{a}.
\end{equation}
\end{theorem}

\begin{proof}
Since $(p,q)\in\mathcal{C}_a$ we know that
$$(p*q)_a(t)=1\ \mbox{for all }t\in\mathbb{N}_{a+1}.$$
It follows that,
$$^{RL}\mathcal{D}_{\{(q);a\}}\mathcal{S}_{\{(p);a\}}f(t)=\Delta(p*(q*f)_a)_a(t)=\Delta(\underbrace{(p*q)_a}_{=1}*f)_a(t)=f(t),\ t\in\mathbb{N}_{a+1}.$$
Now, by the Leibniz rule, we have that
$$^{RL}\mathcal{D}_{\{(q);a\}}f(t)=(q*\Delta f)_a(t)+f(a)q(t)={^{C}\mathcal{D}}_{\{(q);a\}}f(t)+f(a)q(t).$$
Since $\mathcal{S}_{\{(p);a\}}f(a)=0$, we immediately conclude that ${^{C}\mathcal{D}}_{\{(q);a\}}\mathcal{S}_{\{(p);a\}}f(t)=f(t)$ and, consequently, \eqref{proof1} is proved.

To prove \eqref{d1}, we observe that 
\begin{align*}
\mathcal{S}_{\{(p);a\}}{^{RL}\mathcal{D}}_{\{(q);a\}}f(t)&=\mathcal{S}_{\{(p);a\}}[(q*\Delta f)_a(t)+f(a)q(t)]\\
&=(1*\Delta f)_a+f(a)(p*q)_a=f(t).
\end{align*}
Finally, \eqref{d2} follows from 
\begin{align*}
\mathcal{S}_{\{(p);a\}}{^{C}\mathcal{D}}_{\{(q);a\}}f(t)&=\mathcal{S}_{\{(p);a\}}[^{RL}\mathcal{D}_{\{(q);a\}}f-f(a)q]\\
&=f(t)-f(a)(p*q)_a=f(t)-f(a).
\end{align*}
The proof is done.
\end{proof}

%----------------------------------------------------------------------------------------
%   REFERENCE LIST
%----------------------------------------------------------------------------------------
%\phantomsection
%\addcontentsline{toc}{section}{References} % Adds this section to the table of contents

%If you have bibtex file, you can use it here
%\bibliographystyle{apacite}
% \bibliography{references}
%
% or use the following style

\end{document}